\documentclass[11pt]{article}

\usepackage[margin=1.05in]{geometry}
\usepackage{amsmath,amssymb,amsthm,mathtools}
\usepackage{xurl}
\usepackage[colorlinks=true,linkcolor=blue,citecolor=blue,urlcolor=blue]{hyperref}
\usepackage{tikz}

\newtheorem{theorem}{Theorem}[section]
\newtheorem{conjecture}[theorem]{Conjecture}
\newtheorem{proposition}[theorem]{Proposition}

\newtheorem{remark}[theorem]{Remark}

\numberwithin{equation}{section}

\newcommand{\Tr}{\operatorname{Tr}}
\newcommand{\diag}{\operatorname{diag}}
\newcommand{\W}{\mathcal W}
\newcommand{\A}{\mathcal A}
\newcommand{\EA}{E_A}

\title{On the Failure of the Upper Bound in the Refined BMV Conjecture and a Pinching Correction}

\author{
Trung Hoa Dinh\\
Department of Mathematics and Statistics\\
Troy University\\
Troy, AL 36082, USA\\
\texttt{thdinh@troy.edu}
}

\date{}

\begin{document}

\maketitle

\begin{abstract}
We analyze why the refined Bessis--Moussa--Villani conjecture fails.
The refined conjecture proposed that the normalized trace average over all
words with prescribed numbers of letters \(A\) and \(B\) should be bounded
above by the clustered word \(\Tr(A^nB^m)\).  Recent counterexamples of Cha
show that this upper bound is false already for \(3\times3\) positive
semidefinite matrices when \(n=m=5\).  We explain the failure from the
viewpoint of commutative common parts.  The term \(\Tr(A^nB^m)\) is not the
canonical common part of the pair \((A,B)\); it is only one clustered word.
After pinching \(B\) relative to \(A\), the natural commuting contribution is
\(\A_{n,m}(A,\EA(B))\).  The off-diagonal complement \(B-\EA(B)\) creates
spectral bridges, and mixed words can distribute the powers of \(A\) along
closed cycles more efficiently than the clustered word.  This gives a
mechanism for finding counterexamples.  Motivated by this mechanism, we
propose a corrected pinching refinement
\[
        \A_{n,m}(A,B)\ge \A_{n,m}(A,\EA(B)).
\]
We prove this corrected conjecture in the case of two letters \(B\), obtaining
a sandwich refinement
\[
        \A_{n,2}(A,\EA(B))
        \le
        \A_{n,2}(A,B)
        \le
        \Tr(A^nB^2).
\]
Thus, even where the old clustered upper bound remains true, the pinching
viewpoint gives a sharper structural decomposition.
\end{abstract}


\medskip


\section{Introduction}

The Bessis--Moussa--Villani conjecture originated in quantum statistical
mechanics \cite{BMV1975}.  One formulation asserts that, for Hermitian
matrices \(H\) and positive semidefinite matrices \(B\), the function
\[
        t\mapsto \Tr e^{H-tB}
\]
is the Laplace transform of a positive measure on \([0,\infty)\).
Lieb and Seiringer proved that this analytic statement is equivalent to
the positivity of all coefficients of
\[
        \Tr(A+tB)^m,
        \qquad A,B\ge0,\quad m\in\mathbb N,
\]
and to several other formulations \cite{LiebSeiringer2004}.  Stahl proved
the conjecture in \cite{Stahl2013}; see also Eremenko's exposition
\cite{Eremenko2015}.

A later refinement, attributed to H\"agele and listed as Open Quantum
Problem 40 \cite{OQP40}, asks whether the BMV coefficient is not only
nonnegative but also controlled by two distinguished trace expressions.
Earlier work on word traces, sum-of-squares approaches, and partial
developments around the BMV problem includes
\cite{Burgdorf2008,CollinsDykemaTorres2009,LandweberSpeer2009,JohnsonHillar2002}.

For \(n,m\ge0\), let \(\W_{n,m}\) be the set of all words containing exactly
\(n\) letters \(A\) and \(m\) letters \(B\).  Define
\begin{equation}\label{eq:word-average}
        \A_{n,m}(A,B)
        =
        \frac{1}{\binom{n+m}{n}}
        \sum_{W\in\W_{n,m}}\Tr W(A,B).
\end{equation}
Equivalently, \(\binom{n+m}{n}\A_{n,m}(A,B)\) is the coefficient of
\(t^m\) in \(\Tr(A+tB)^{n+m}\).

The refined conjecture asks whether
\begin{equation}\label{eq:refined-bmv}
        \Tr \exp(n\log A+m\log B)
        \le
        \A_{n,m}(A,B)
        \le
        \Tr(A^nB^m)
\end{equation}
for positive definite matrices \(A,B\), with the positive semidefinite
case obtained by approximation.

The upper bound in \eqref{eq:refined-bmv} was recently disproved by Cha
\cite{Cha2026}.  Cha constructed a one-parameter family of positive
semidefinite \(3\times3\) matrices \((A_x,B_x)\) for which
\[
        \A_{5,5}(A_x,B_x)>\Tr(A_x^5B_x^5)
\]
for small \(x>0\), for instance \(x=10^{-3}\).  Moreover, the ratio $\frac{\A_{5,5}(A_x,B_x)}
             {\Tr(A_x^5B_x^5)}$
can become arbitrarily large as \(x\to0\).

The goal of this paper is to explain why the upper bound in
\eqref{eq:refined-bmv} should fail.  The central point is that
\(\Tr(A^nB^m)\) is not the canonical commutative common part of the word
average.  It is merely one clustered word.  A canonical common part is
obtained by keeping \(A\) fixed and replacing \(B\) by the part visible in
the commutant of \(A\).

Let $A=\sum_{\lambda}\lambda P_\lambda$ 
be the spectral decomposition of \(A\).  We define the pinching of \(B\)
relative to \(A\) by
\begin{equation}\label{eq:pinching}
        \EA(B)=\sum_{\lambda}P_\lambda B P_\lambda .
\end{equation}
Then \(\EA(B)\) commutes with \(A\).  Hence
\[
        \A_{n,m}(A,\EA(B))
        =
        \Tr\bigl(A^n\EA(B)^m\bigr).
\]
This is the part of the word average already explained by the commuting
algebra seen by \(A\).

The refined upper bound compares the word average with
\(\Tr(A^nB^m)\), which mixes the pinched part \(\EA(B)\) and the
off-diagonal complement \(B-\EA(B)\).  Thus it does not subtract the common
direction.  The off-diagonal complement can create positive closed cycles,
which we call spectral bridges.  These bridge contributions may appear in
mixed words with a lower spectral cost than in the clustered word.  This
is the mechanism behind the counterexamples.

The corrected principle suggested here is the following pinching
refinement:
\begin{equation}\label{eq:pinching-conjecture-intro}
        \A_{n,m}(A,B)\ge \A_{n,m}(A,\EA(B)).
\end{equation}
This reverses the sign direction from the failed clustered upper bound.
The averaged word trace should dominate its pinched commuting part, not be
dominated by a single clustered word.

We prove \eqref{eq:pinching-conjecture-intro} in the case \(m=2\).  In
fact, in that case one has the sharper sandwich
\[
        \A_{n,2}(A,\EA(B))
        \le
        \A_{n,2}(A,B)
        \le
        \Tr(A^nB^2).
\]
Thus the old clustered upper bound is still true for two letters \(B\), but
the pinching viewpoint gives a more precise structural decomposition.

\medskip

The paper is organized as follows.  Section~2 explains why the clustered
word \(\Tr(A^nB^m)\) is the wrong common part and introduces the pinched
common part \(\A_{n,m}(A,E_A(B))\).  Section~3 describes the
spectral-bridge mechanism that produces counterexamples to the refined
upper bound.  Section~4 reinterprets Cha's counterexample through this
pinching decomposition.  Section~5 formulates the corrected pinching
conjecture.  Section~6 proves the conjecture for the case of two letters
\(B\), giving a sandwich refinement. The final section discusses the higher-cycle phase obstructions that arise
beyond two letters \(B\).

\section{The refined upper bound and the wrong common part}

The upper part of the refined BMV conjecture is
\begin{equation}\label{eq:old-upper}
        \A_{n,m}(A,B)\le \Tr(A^nB^m).
\end{equation}
Equivalently,
\begin{equation}\label{eq:old-remainder}
        R_{n,m}(A,B)
        :=
        \A_{n,m}(A,B)-\Tr(A^nB^m)
        \le0.
\end{equation}

At first sight \eqref{eq:old-upper} looks natural.  If \(A\) and \(B\)
commute, then every word \(W\in\W_{n,m}\) satisfies
\[
        W(A,B)=A^nB^m.
\]
Hence
\[
        \A_{n,m}(A,B)=\Tr(A^nB^m).
\]
Thus \(\Tr(A^nB^m)\) is the value obtained after collapsing the whole word
ensemble to the commuting case.

However, this also reveals the weakness of \eqref{eq:old-upper}.  The
quantity \(\Tr(A^nB^m)\) is not a canonical noncommutative upper anchor.
It is only one clustered word.  In the noncommutative setting, there is no
structural reason that this word should dominate the average of all words.

A more intrinsic way to extract the commuting part is to keep \(A\) fixed
and pinch \(B\) relative to \(A\).  With \(\EA(B)\) as in
\eqref{eq:pinching}, we have
\[
        A\EA(B)=\EA(B)A.
\]
Therefore
\begin{equation}\label{eq:pinched-average}
        \A_{n,m}(A,\EA(B))
        =
        \Tr\bigl(A^n\EA(B)^m\bigr).
\end{equation}
This is the genuine commuting contribution attached to the pair \((A,B)\).

The decomposition
\begin{equation}\label{eq:B-decomposition}
        B=\EA(B)+B_\perp,
        \qquad
        B_\perp:=B-\EA(B),
\end{equation}
separates the common part from the off-diagonal complement.  The clustered
word \(\Tr(A^nB^m)\), however, does not respect this separation.  Since
\(B^m\) mixes \(\EA(B)\) with \(B_\perp\), the expression $\Tr(A^nB^m)$
already contains noncommutative bridge terms.  It is not the pure
commuting contribution.

The correct decomposition should begin from
\begin{equation}\label{eq:correct-decomp}
        \A_{n,m}(A,B)
        =
        \A_{n,m}(A,\EA(B))
        +
        \Bigl[
        \A_{n,m}(A,B)-\A_{n,m}(A,\EA(B))
        \Bigr].
\end{equation}
The first term is the pinched commuting part.  The second term is the
noncommutative gap.

Thus the natural sign question is not
\[
        \A_{n,m}(A,B)-\Tr(A^nB^m)\le0,
\]
but rather
\begin{equation}\label{eq:new-gap}
        \A_{n,m}(A,B)-\A_{n,m}(A,\EA(B))\ge0.
\end{equation}
In this form the direction is reversed.  The average should dominate its
pinched common part, not be dominated by a clustered word.

This explains why the refined upper bound is structurally unstable.  It
uses the wrong object as the common part.

\section{Spectral bridges and the mechanism for counterexamples}

We now make the preceding discussion more concrete.  Diagonalize \(A\):
\[
        A=\diag(a_1,\ldots,a_d),
        \qquad a_i\ge0.
\]
Write
\[
        B=D+N,
        \qquad D=\EA(B),
        \qquad N=B-\EA(B).
\]
In this basis, \(D\) is block diagonal along the eigenspaces of \(A\), while
\(N\) contains transitions between distinct eigenspaces.

The clustered word is $ \Tr(A^nB^m).$ 
Expanding in the eigenbasis of \(A\), a closed path
\[
        i_0\to i_1\to\cdots\to i_m=i_0
\]
contributes
\begin{equation}\label{eq:clustered-path}
        a_{i_0}^{\,n}
        b_{i_0i_1}b_{i_1i_2}\cdots b_{i_{m-1}i_0}.
\end{equation}
Thus the full \(A^n\)-weight is paid at a single vertex.

By contrast, a general cyclic representative of a word in the average has
the form
\[
        A^{r_1}BA^{r_2}B\cdots A^{r_m}B,
        \qquad
        r_1+\cdots+r_m=n.
\]
A closed path
\[
        i_1\to i_2\to\cdots\to i_m\to i_1
\]
then contributes
\begin{equation}\label{eq:mixed-path}
        a_{i_1}^{r_1}a_{i_2}^{r_2}\cdots a_{i_m}^{r_m}
        b_{i_1i_2}b_{i_2i_3}\cdots b_{i_mi_1}.
\end{equation}
The powers of \(A\) can now be distributed around the cycle.

This difference is the source of the failure of the upper bound.  The
clustered word forces the entire \(A\)-weight onto one vertex, whereas the
word average allows the \(A\)-weight to be distributed among several
vertices.  If \(A\) has a highly anisotropic spectrum, this can change the
order of magnitude of the trace.

Thus a counterexample to \eqref{eq:old-upper} can be sought by the
following mechanism.

\medskip

\noindent
\textbf{Step 1. Choose a highly anisotropic spectrum.}
Take
\[
        A=\diag(a_1,\ldots,a_d),
        \qquad
        a_1\gg a_2\gg\cdots\gg a_d\ge0.
\]
Then the order of a bridge contribution is extremely sensitive to where the
powers of \(A\) are placed.  In small-parameter examples, concentrating many
powers on a small spectral vertex can strongly suppress the clustered word,
whereas mixed words may move some of the \(A\)-powers to larger spectral
vertices and produce a larger contribution.

\medskip

\noindent
\textbf{Step 2. Make the pinched part small on the large eigenspace.}
Write
\[
        B=D+N,
        \qquad
        D=\EA(B),
        \qquad
        N=B-\EA(B).
\]
Choose \(B\) so that the block of \(D\) on the large eigenspace of \(A\)
is small.  Then the pinched common part
\[
        \A_{n,m}(A,D)
\]
is small along the dominant spectral direction.

\medskip

\noindent
\textbf{Step 3. Create off-diagonal bridges.}
Choose the complement \(N=B-\EA(B)\) so that it has off-diagonal entries
connecting different eigenspaces of \(A\).  These entries create closed
cycles such as
\[
        i_1\to i_2\to\cdots\to i_m\to i_1,
\]
whose contribution contains products of the form
\[
        n_{i_1i_2}n_{i_2i_3}\cdots n_{i_mi_1}.
\]
If the real part of such a product is positive, the complement produces
a positive spectral bridge.

\medskip

\noindent
\textbf{Step 4. Compare clustered and mixed spectral costs.}
The clustered word $ \Tr(A^nB^m)$ 
forces a path contribution of the form
\[
        a_{i_0}^{\,n}
        b_{i_0i_1}b_{i_1i_2}\cdots b_{i_{m-1}i_0}.
\]
Thus all \(A^n\)-weight is paid at a single vertex.

By contrast, a mixed word
\[
        A^{r_1}BA^{r_2}B\cdots A^{r_m}B,
        \qquad
        r_1+\cdots+r_m=n,
\]
has path contribution
\[
        a_{i_1}^{r_1}a_{i_2}^{r_2}\cdots a_{i_m}^{r_m}
        b_{i_1i_2}b_{i_2i_3}\cdots b_{i_mi_1}.
\]
Here the \(A\)-weight is distributed along the cycle.  For a highly
anisotropic spectrum, this distributed configuration may have lower
spectral cost than the clustered word.

\medskip

Consequently, if the off-diagonal bridges have positive contribution and
some mixed words have lower spectral cost than the clustered word, then
the bridge-gain may dominate:
\[
        \A_{n,m}(A,B)>\Tr(A^nB^m).
\]
This gives a counterexample to the refined upper bound.

This mechanism is consistent with Cha's analysis.  Cha shows that the
leading behavior of the counterexample is governed not by a crude number
of alternations, but by a weighted shortest-bridge cost on cyclic run
decompositions \cite{Cha2026}.  In the language used here, the
off-diagonal complement \(B-\EA(B)\) creates spectral bridges whose
contribution survives averaging.

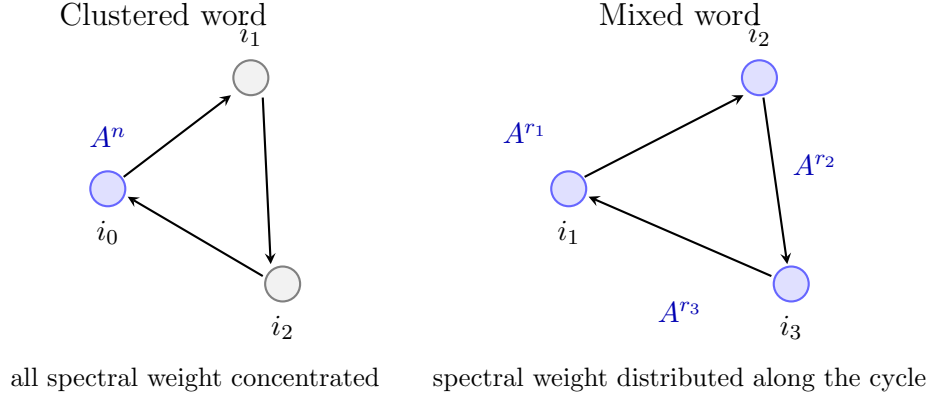
\begin{figure}[t]
\centering

\begin{tikzpicture}[>=stealth,scale=1.05]


\node at (-3.3,2.2) {\large Clustered word};

\filldraw[fill=blue!12,draw=blue!60,thick]
(-4,0) circle (0.22);

\filldraw[fill=gray!10,draw=black!50,thick]
(-2.2,1.4) circle (0.22);

\filldraw[fill=gray!10,draw=black!50,thick]
(-1.8,-1.2) circle (0.22);

\node at (-4,-0.55) {$i_0$};
\node at (-2.2,1.95) {$i_1$};
\node at (-1.8,-1.75) {$i_2$};

\draw[->,thick]
(-3.8,0.15)--(-2.45,1.18);

\draw[->,thick]
(-2.05,1.15)--(-1.95,-0.95);

\draw[->,thick]
(-2.05,-1.1)--(-3.75,-0.1);

\node[blue!70!black] at (-4,0.65) {$A^n$};

\node at (-2.9,-2.4)
{\small all spectral weight concentrated};


\node at (3.2,2.2) {\large Mixed word};

\filldraw[fill=blue!12,draw=blue!60,thick]
(1.8,0) circle (0.22);

\filldraw[fill=blue!12,draw=blue!60,thick]
(4.2,1.4) circle (0.22);

\filldraw[fill=blue!12,draw=blue!60,thick]
(4.6,-1.2) circle (0.22);

\node at (1.8,-0.55) {$i_1$};
\node at (4.2,1.95) {$i_2$};
\node at (4.6,-1.75) {$i_3$};

\draw[->,thick]
(2.0,0.15)--(4.0,1.18);

\draw[->,thick]
(4.25,1.15)--(4.55,-0.95);

\draw[->,thick]
(4.35,-1.1)--(2.05,-0.1);

\node[blue!70!black] at (1.25,0.7) {$A^{r_1}$};
\node[blue!70!black] at (4.9,0.3) {$A^{r_2}$};
\node[blue!70!black] at (3.2,-1.55) {$A^{r_3}$};

\node at (3.2,-2.4)
{\small spectral weight distributed along the cycle};


\node[align=center]
at (0,-3.7)
{
\small
Clustered words force the entire \(A\)-weight onto one spectral vertex,\\
while mixed words distribute the spectral cost along closed bridges.
};

\end{tikzpicture}

\caption{
Spectral-bridge mechanism behind the failure of the refined upper bound.
The clustered word \(\Tr(A^nB^m)\) concentrates the full \(A^n\)-weight at one
vertex, whereas mixed words distribute the powers of \(A\) along a closed
cycle.  For highly anisotropic spectra, the distributed configuration may
have substantially lower spectral cost.
}
\label{fig:spectral-bridge}
\end{figure}

\section{Cha's counterexample through pinching}

Cha's counterexample uses the family
\begin{equation}\label{eq:cha-family}
A_x=
\begin{pmatrix}
1&0&0\\
0&x&-x\\
0&-x&x
\end{pmatrix},
\qquad
B_x=
\begin{pmatrix}
x&-x&0\\
-x&x&0\\
0&0&1
\end{pmatrix},
\qquad x>0.
\end{equation}
Both matrices are positive semidefinite.  Their eigenvalues are
\[
        0,\quad 2x,\quad 1.
\]
Thus \(A_x\) is strongly anisotropic when \(x\) is small.

In the eigenbasis of \(A_x\), the pinching of \(B_x\) is
\begin{equation}\label{eq:pinching-cha}
        E_{A_x}(B_x)
        =
        \diag\left(
        x,\frac{1+x}{2},\frac{1+x}{2}
        \right),
\end{equation}
while the eigenvalues of \(A_x\) are
\[
        1,\quad 2x,\quad 0.
\]
Therefore
\begin{equation}\label{eq:cha-pinched-value}
        \A_{5,5}(A_x,E_{A_x}(B_x))
        =
        \Tr\bigl(A_x^5E_{A_x}(B_x)^5\bigr)
        =
        x^5\bigl(1+(1+x)^5\bigr).
\end{equation}
At \(x=10^{-3}\),
\begin{equation}\label{eq:cha-pinched-numeric}
        \A_{5,5}(A_x,E_{A_x}(B_x))
        \approx
        2.0050100100\times10^{-15}.
\end{equation}

On the other hand,
\begin{equation}\label{eq:cha-clustered}
        \Tr(A_x^5B_x^5)
        =
        32x^5+256x^{10}.
\end{equation}
Hence, at \(x=10^{-3}\),
\begin{equation}\label{eq:cha-clustered-numeric}
        \Tr(A_x^5B_x^5)
        \approx
        3.2000000000\times10^{-14}.
\end{equation}

A direct expansion gives the exact formula
\begin{equation}\label{eq:cha-average-exact}
        \A_{5,5}(A_x,B_x)
        =
        \frac{x^4}{126}
        \bigl(
        5+1422x+1675x^2+3130x^3
        +4875x^4+5930x^5+4881x^6
        \bigr).
\end{equation}
Hence, at \(x=10^{-3}\),
\begin{equation}\label{eq:cha-average-numeric}
        \A_{5,5}(A_x,B_x)
        \approx
        5.0981572499\times10^{-14}.
\end{equation}
Moreover,
\[
        \A_{5,5}(A_x,B_x)\sim \frac{5}{126}x^4,
        \qquad
        \Tr(A_x^5B_x^5)\sim 32x^5,
\]
so
\[
        \frac{\A_{5,5}(A_x,B_x)}{\Tr(A_x^5B_x^5)}
        \sim
        \frac{5}{4032}\frac1x
        \qquad (x\to0).
\]
Consequently,
\begin{equation}\label{eq:cha-ordering}
        \A_{5,5}(A_x,E_{A_x}(B_x))
        <
        \Tr(A_x^5B_x^5)
        <
        \A_{5,5}(A_x,B_x)
\end{equation}
for sufficiently small \(x>0\), and in particular for \(x=10^{-3}\).

Thus Cha's example destroys the clustered upper bound but does not
contradict the pinching principle.  Rather, it supports the idea that the
clustered word is not the correct common part.  The pinched term records
the commuting contribution.  The off-diagonal complement
\[
        B_x-E_{A_x}(B_x)
\]
creates spectral bridges, and the mixed words capture more of this
bridge-gain than the clustered word does.

\section{The corrected pinching conjecture}

The previous sections suggest that the correct refinement should not be an
upper bound by a clustered word.  It should be a lower bound by the pinched
commuting part.

\begin{conjecture}[Pinching refinement]\label{conj:pinching}
Let \(A,B\ge0\).  Then, for all \(n,m\ge0\),
\begin{equation}\label{eq:pinching-refinement}
        \A_{n,m}(A,B)
        \ge
        \A_{n,m}(A,\EA(B)).
\end{equation}
Equivalently,
\begin{equation}\label{eq:pinching-refinement-equivalent}
        \A_{n,m}(A,B)
        \ge
        \Tr\bigl(A^n\EA(B)^m\bigr).
\end{equation}
\end{conjecture}

Since \(A\) commutes with \(\EA(B)\), the right-hand side is nonnegative.
Therefore Conjecture \ref{conj:pinching} implies the coefficient positivity
part of the BMV theorem:
\[
        \A_{n,m}(A,B)\ge0.
\]

The conceptual difference from the failed upper bound is important.  The
failed inequality compares the word average with one clustered word.  The
pinching conjecture compares the word average with the part already visible
in the commutative algebra generated by \(A\).

\section{\texorpdfstring{The case of two \(B\)'s: a sandwich refinement}
{The case of two B's: a sandwich refinement}}

Before reaching the range where the clustered upper bound fails, it is
useful to examine the first nontrivial case, namely the case of two letters
\(B\).  In this case the old clustered upper bound is still valid, but the
pinching viewpoint gives a sharper structural statement.

For \(n\ge0\), define
\begin{equation}\label{eq:h-n}
        h_n(x,y)=\sum_{r=0}^{n}x^r y^{n-r}.
\end{equation}

\begin{proposition}[Sandwich for two \(B\)'s]\label{prop:sandwich-two-B}
Let \(A,B\ge0\).  Then, for every \(n\ge0\),
\begin{equation}\label{eq:sandwich-two-B}
        \A_{n,2}(A,\EA(B))
        \le
        \A_{n,2}(A,B)
        \le
        \Tr(A^nB^2).
\end{equation}
\end{proposition}

\begin{proof}
By unitary invariance, assume first that
\[
        A=\diag(a_1,\ldots,a_d),
        \qquad a_i\ge0,
\]
and write \(B=(b_{ij})\).

Taking cyclicity of the trace into account, the full sum over all words with
\(n\) letters \(A\) and two letters \(B\) satisfies
\[
        \sum_{W\in\W_{n,2}}\Tr W(A,B)
        =
        \frac{n+2}{2}
        \sum_{r=0}^{n}\Tr(A^rBA^{n-r}B).
\]
Indeed, the two gaps between the two occurrences of \(B\) are \(r\) and
\(n-r\).  The two orientations are represented by the terms \(r\) and
\(n-r\), and the symmetric case \(r=n-r\) gives the same factor by cyclic
periodicity.  Hence
\begin{align}
        \A_{n,2}(A,B)
        &=
        \frac{1}{\binom{n+2}{2}}
        \frac{n+2}{2}
        \sum_{r=0}^{n}\Tr(A^rBA^{n-r}B)                         \notag\\
        &=
        \frac{1}{n+1}
        \sum_{i,j=1}^{d}
        h_n(a_i,a_j)|b_{ij}|^2 .                                \label{eq:A-n2-expansion}
\end{align}

Since \(\EA(B)\) keeps only the blocks of \(B\) inside the eigenspaces of
\(A\), the difference
\[
        \A_{n,2}(A,B)-\A_{n,2}(A,\EA(B))
\]
is precisely the contribution of the off-block entries:
\begin{equation}\label{eq:positive-gap-two-B}
        \A_{n,2}(A,B)-\A_{n,2}(A,\EA(B))
        =
        \frac{1}{n+1}
        \sum_{a_i\ne a_j}
        h_n(a_i,a_j)|b_{ij}|^2
        \ge0.
\end{equation}
This proves the lower bound in \eqref{eq:sandwich-two-B}.

For the upper bound, note that
\begin{equation}\label{eq:clustered-n2}
        \Tr(A^nB^2)
        =
        \sum_{i,j=1}^{d}a_i^n|b_{ij}|^2.
\end{equation}
Pairing the terms \((i,j)\) and \((j,i)\), it is enough to prove
\begin{equation}\label{eq:scalar-upper-two-B}
        \frac{2}{n+1}h_n(a_i,a_j)
        \le
        a_i^n+a_j^n.
\end{equation}
For \(n=0\), this is equality.  For \(n\ge1\), the weighted AM--GM
inequality gives
\[
        a_i^r a_j^{n-r}
        \le
        \frac{r}{n}a_i^n+\frac{n-r}{n}a_j^n,
        \qquad r=0,\ldots,n.
\]
Summing over \(r\) gives
\[
        \frac{1}{n+1}h_n(a_i,a_j)
        \le
        \frac{a_i^n+a_j^n}{2}.
\]
Thus \eqref{eq:scalar-upper-two-B} holds, and hence
\[
        \A_{n,2}(A,B)\le \Tr(A^nB^2).
\]

If \(A\) has repeated eigenvalues, choose a basis inside each eigenspace of
\(A\) that diagonalizes the compression \(P_\lambda B P_\lambda\).  The
same computation applies, with off-block entries corresponding to distinct
eigenvalues.  This completes the proof.
\end{proof}

\begin{remark}
Proposition \ref{prop:sandwich-two-B} shows that the pinching refinement is
not merely a repair after the failure of the clustered upper bound.  Even
in a range where the old clustered upper bound remains true, the pinching
viewpoint gives more information.  It identifies the exact commuting
contribution and gives the positive square gap
\[
        \frac{1}{n+1}
        \sum_{a_i\ne a_j}
        h_n(a_i,a_j)|b_{ij}|^2 .
\]
Thus the old bound is an upper estimate, while the pinching refinement is a
structural decomposition.
\end{remark}

\section{Discussion}

The original BMV theorem is a positivity theorem for the total coefficient
of \(\Tr(A+tB)^m\).  The refined upper bound \eqref{eq:old-upper} attempted
to strengthen this by comparing the full word average with a single
clustered word.  Cha's counterexamples show that this word-level domination
is false \cite{Cha2026}.

The failure is structural.  The clustered word is not the common part of
the word average.  The canonical common part, relative to \(A\), is obtained
by pinching \(B\):
\[
        B=\EA(B)+\bigl(B-\EA(B)\bigr).
\]
The part \(\EA(B)\) gives the commuting contribution.  The off-diagonal
part \(B-\EA(B)\) creates spectral bridges between eigenspaces of \(A\).
Mixed words can distribute the powers of \(A\) along these bridges and may
therefore have lower spectral cost than the clustered word.

The case of two letters \(B\) shows why this phenomenon is not visible at
the first nontrivial level.  In that case the expansion contains only edge
terms of the form $ |b_{ij}|^2.$ 
No cyclic phase obstruction is present, and this is why both the pinching
lower bound and the old clustered upper bound can be proved directly in
Proposition \ref{prop:sandwich-two-B}.

Starting from three letters \(B\), however, closed cycles appear.  For
example, words with three letters \(B\) contain terms of the form
\[
        b_{ij}b_{j\ell}b_{\ell i}.
\]
The real part of such a product can have either sign.  These higher-cycle
terms are precisely the spectral-bridge effects that are absent in the
two-\(B\) case.  Thus the general pinching conjecture requires controlling
cyclic phase contributions by the positivity of \(B\).

The corrected viewpoint is therefore to subtract the pinched common part
instead of comparing the word average with one clustered word.  The failed
upper bound asks for
\[
        \A_{n,m}(A,B)-\Tr(A^nB^m)\le0,
\]
which is not structurally natural.  The pinching refinement asks instead
for
\[
        \A_{n,m}(A,B)-\A_{n,m}(A,\EA(B))\ge0.
\]
This is a genuine noncommutative gap above the commuting component visible
from \(A\).  In this sense, the counterexamples to the refined upper bound
do not merely disprove a proposed inequality; they point to the correct
object that should be subtracted.



\begin{thebibliography}{99}

\bibitem{BMV1975}
D. Bessis, P. Moussa, and M. Villani,
Monotonic converging variational approximations to the functional integrals
in quantum statistical mechanics,
\emph{Journal of Mathematical Physics} \textbf{16} (1975), 2318--2325.

\bibitem{LiebSeiringer2004}
E. H. Lieb and R. Seiringer,
Equivalent forms of the Bessis--Moussa--Villani conjecture,
\emph{Journal of Statistical Physics} \textbf{115} (2004), 185--190.

\bibitem{Stahl2013}
H. R. Stahl,
Proof of the BMV conjecture,
\emph{Acta Mathematica} \textbf{211} (2013), 255--290.

\bibitem{Eremenko2015}
A. Eremenko,
Herbert Stahl's proof of the BMV conjecture,
\emph{Sbornik: Mathematics} \textbf{206} (2015), 87--92.

\bibitem{OQP40}
IQOQI Vienna,
Open Quantum Problem 40: Refinement of the Bessis--Moussa--Villani conjecture.
Available at
\url{https://oqp.iqoqi.oeaw.ac.at/refinement-of-the-bessis-moussa-villani-conjecture}.

\bibitem{Cha2026}
H. Cha,
One-parameter counterexamples to the refined Bessis--Moussa--Villani conjecture,
arXiv:2603.19927, 2026.
Available at
\url{https://arxiv.org/abs/2603.19927}.

\bibitem{Burgdorf2008}
S. Burgdorf,
Sums of Hermitian squares as an approach to the BMV conjecture,
arXiv:0802.1153.

\bibitem{CollinsDykemaTorres2009}
B. Collins, K. J. Dykema, and F. Torres-Ayala,
Sum-of-squares results for polynomials related to the
Bessis--Moussa--Villani conjecture,
arXiv:0905.0420.

\bibitem{LandweberSpeer2009}
P. S. Landweber and E. R. Speer,
On D. H\"agele's approach to the Bessis--Moussa--Villani conjecture,
\emph{Linear Algebra and its Applications} \textbf{431} (2009), 1317--1324.

\bibitem{JohnsonHillar2002}
C. R. Johnson and C. J. Hillar,
Eigenvalues of words in two positive definite letters,
\emph{SIAM Journal on Matrix Analysis and Applications} \textbf{23} (2002),
916--928.

\end{thebibliography}
\end{document}